\documentclass[final]{siamltex}
\usepackage{amsfonts}
\usepackage{amsmath}
\usepackage{mathtools}
\usepackage{latexsym}
\usepackage{amssymb}
\usepackage{booktabs}
\usepackage{array}

\usepackage{graphicx,overpic,xcolor}
\usepackage[list=true, 
labelformat=brace, position=top]{subcaption}
\graphicspath{{./pictures/}{./asymptote/}}
\DeclareGraphicsExtensions{.png,.pdf}

\usepackage[english=american]{csquotes}
\usepackage{hyperref}
\hypersetup{
    colorlinks=true, 
    linktoc=all,     
    linkcolor=blue,  
    citecolor=red,
}
\usepackage{algorithm}
\usepackage{algpseudocode}
\usepackage{setspace}
\usepackage{pgfplots}
\pgfplotsset{compat=newest}
\usetikzlibrary{patterns,backgrounds}
\definecolor{rwthBlue}{RGB}{0,84,159}
\definecolor{rwthLightBlue}{RGB}{142,186,229}
\definecolor{rwthOrange}{RGB}{246,168,0}
\definecolor{rwthGreen}{RGB}{87,171,39}
\definecolor{rwthDarkGreen}{RGB}{0,97,101}
\definecolor{rwthLightGreen}{RGB}{189,205,0}
\definecolor{rwthCyan}{RGB}{0,152,161}
\definecolor{rwthRed}{RGB}{204,7,30}
\definecolor{rwthDarkRed}{RGB}{161,16,53}
\definecolor{rwthLila}{RGB}{122,111,172}
\definecolor{rwthPurple}{RGB}{97,33,88}
\definecolor{rwthYellow}{RGB}{255,237,0}
\definecolor{IGPMred}{RGB}{156,12,16}
\definecolor{IGPMblue}{RGB}{0,105,164}

\newcounter{assum}

\newcommand{\bu} {\mathbf{u}}

\newcommand{\bA} {\mathbf{A}}

\newcommand{\R} {\mathbb{R}}
\newcommand{\T} {\mathcal{T}}
\newcommand{\J} {\mathcal{J}}
\newcommand{\I} {\mathcal{I}}
\newcommand{\F} {\mathcal{F}}
\newcommand{\Pol} {\mathcal{P}}
\newcommand{\bx} {\mathbf{x}}

\newcommand{\by} {\mathbf{y}}

\newcommand{\bn} {\mathbf{n}}

\newcommand{\bv} {\mathbf{v}}

\newtheorem{remark}{Remark}

\newcommand{\jumpleft}{[\![}
\newcommand{\jumpright}{]\!]}
\newcommand{\jump}[1]{\jumpleft #1 \jumpright}

\newcommand{\cTG}{\mathcal{T}_h^\Gamma}
\newcommand{\OG}{\Omega_h^\Gamma}

\newcommand{\MAT}[1]{\mathbf{#1}}

\def\b|{{\|\hskip-0.16ex|}}

\title{Optimal preconditioners for a Nitsche stabilized fictitious domain finite element method}

\author{Sven Gro{\SS}\thanks{Institut f\"ur
Geometrie und Praktische  Mathematik, RWTH Aachen University, D-52056 Aachen,
Germany; email: {\tt gross@igpm.rwth-aachen.de}}
\and Arnold Reusken\thanks{Institut f\"ur
Geometrie und Praktische  Mathematik, RWTH Aachen University, D-52056 Aachen,
Germany; email: {\tt reusken@igpm.rwth-aachen.de}} 
 }

\begin{document}

\maketitle

\begin{abstract}
In this paper we consider a class of fictitious domain finite element methods known from the literature. These methods use standard finite element spaces on a fixed unfitted triangulation combined with the Nitsche technique and a ghost penalty stabilization. As a model problem we consider the application of such a method to the Poisson equation. We introduce and analyze a new class of preconditioners  that is based on a subspace decomposition approach. The finite element space is split into two subspaces, where one subspace is spanned by all nodal basis functions corresponding to nodes on the boundary of the fictitious domain and the other space is spanned by all remaining nodal basis functions. We will show that this splitting is stable, uniformly in the discretization parameter  and in the location of the problem boundary in the triangulation. We also prove that the Galerkin discretization in the first subspace leads to a uniformly well-conditioned matrix and that the Galerkin discretization in the second subspace  is uniformly equivalent to a standard finite element discretization of a Poisson equation on the fictitious domain with homogeneous Dirichlet boundary conditions. Results of numerical experiments that illustrate optimality of such a  preconditioner are included.
\end{abstract}


\section{Introduction}
In recent years many papers appeared in which the so-called CutFEM paradigm is developed and  analyzed, cf. the overview references \cite{burman2015cutfem,CutFEM}. In this approach for discretization of a partial differential equation a fixed \emph{unfitted mesh} is used that is not aligned with a (moving) interface and/or a complex domain boundary. On this mesh \emph{standard finite element spaces} are used. For treating the boundary and/or interface conditions, either a Lagrange multiplier technique or Nitsche's method is applied. In the setting of the present paper we restrict to \emph{Nitsche's method}. Furthermore, to avoid ill-conditioning of the resulting discrete systems (due to ``small cuts'') a stabilization technique is used. The most often used approach is the \emph{ghost-penalty stabilization}. In the literature the different components of this general technique are studied, error analyses are presented and different fields of applications are studied \cite{burman2015cutfem,CutFEM}. Related unfitted finite element methods are  popular in fracture mechanics \cite{fries2010extended}; in that community these are often called extended finite element methods (XFEM).

Almost all papers on CutFEM (or XFEM) either treat  applications of this methodology or present discretization error analyses. In relatively few papers efficient solvers for the resulting discrete problems are studied. 
In \cite{burmanhansbo12,zawakrbe13,HaZa2014} for the resulting stiffness matrix condition  number bounds of the form $ch^{-2}$, with a constant $c$ that is independent of how an interface or boundary intersects the triangulation, have been derived. In \cite{burmanhansbo12} a fictitious domain variant of CutFEM is introduced and it is shown that discretization of a Poisson equation using this method yields a stiffness matrix with such a condition number bound. In \cite{zawakrbe13} a similar result is derived for CutFEM applied to  a Poisson interface problem. In \cite{HaZa2014}  a condition number bound is derived  for CutFEM applied to a Stokes interface problem. These papers do \emph{not} treat efficient preconditioners for the stiffness matrix. 

There are  few papers in which (multigrid type) efficient preconditioners for CutFEM or closely related discretizations (e.g., XFEM) are treated, e.g.,  \cite{berger2012inexact,badia2017robust,de2017preconditioning,jo2018geometric,de2020preconditioning,ludescher1,ludescher}.
In none of these papers a rigorous analysis of the spectral quality of the preconditioner is presented. The only paper that we know of that contains such a rigorous analysis is \cite{lehrenfeld2017optimal}. In that paper a CutFEM \emph{without} stabilization is analyzed for a two-dimensional Poisson interface problem.

The main topic of the present paper is an analysis of a (new) subspace decomposition based preconditioner that is optimal (in a sense explained in section~\ref{sectPrecond}) for a CutFEM  fictitious domain method as in \cite{burmanhansbo10,burmanhansbo12,Massing2014}. We expect that similar preconditioners can be developed and rigorously analyzed for other CutFEM applications such as a Stokes fictitious domain method, or Poisson and Stokes interface problems.

  We explain the key idea of the preconditioner. In the setting of a fictitious domain approach the finite element space is split into two subspaces. One subspace is spanned by all nodal basis functions corresponding to nodes on the boundary of the fictitious domain and the other space is spanned by all remaining nodal basis functions. We will show that this splitting is stable, uniformly in the discretization parameter $h$ and in the location of the true boundary in the triangulation. We also prove that the Galerkin discretization in the first subspace leads to a uniformly well-conditioned matrix and that the Galerkin discretization in the second subspace  is uniformly equivalent to a standard finite element discretization of a Poisson equation on the fictitious domain with homogeneous Dirichlet boundary conditions. Using this property it can be shown that a multigrid method yields an optimal preconditioner for the Galerkin discretization in the second subspace.  An additive Schwarz subspace correction method (or, equivalently, block Jacobi) yields an optimal preconditioner for the CutFEM fictitious domain discretization.
  
  We briefly address important differences between the  results in this paper and in \cite{lehrenfeld2017optimal}. In the latter a CutFEM variant  \emph{without}  stabilization is used which is applied to a Poisson interface problem. The preconditioner is based on a subspace splitting that is similar to the one studied in this paper. The analysis in  \cite{lehrenfeld2017optimal} is restricted to linear finite elements and two-dimensional problems. In this paper we consider the CutFEM \emph{with stabilization}. It turns out that this allows an elegant, rather simple  and much more general analysis. In particular, the analysis covers two- and three-dimensional problems, arbitrary polynomial degree finite elements and triangulations that are shape regular but not necessarily quasi-uniform.

The paper is organized as follows.   In Section~\ref{sec:disc} we describe a CutFEM fictitious domain method known from the literature. In Section~\ref{sectsubspace} we introduce and analyze a natural splitting of the finite element space. Based on this stable splitting we propose (optimal) preconditioners in Section~\ref{sectPrecond}. In Section~\ref{sectNumExp} results of numerical experiments with these preconditioners are presented.

\section{CutFEM: a stabilized Nitsche fictitious domain method} \label{sec:disc}
We recall a fictitious domain method known from the literature \cite{burmanhansbo12,Massing2014}. We restrict to the simple setting of the Poisson equation:
\begin{equation}\label{Poisson}
\begin{split}
 -\Delta u  &= f \quad \text{in } \Omega,\\
 u &= g  \quad \text{on } \Gamma:=\partial\Omega.
\end{split}
\end{equation}
Here $\Omega \subset \R^d$, $d=2,3$, is an open connected Lipschitz domain. 
For simplicity we take Dirichlet boundary conditions on the whole boundary. The method and analysis below are easily extended to the case with Dirichlet boundary conditions on only part of the boundary and to more general symmmetric elliptic boundary value problems. \\
We take a larger polygonal domain $\Omega^\ast \supset \Omega$ and a family of shape regular simplicial triangulations $\{\T_h^\ast\}_{h >0}$ of the larger domain $\Omega^\ast$. The set of simplices that cut the domain $\Omega$ and the corresponding \emph{fictitious domain} are defined by
\begin{equation} \label{Pa}
  \T_h:=\{\, T \in \T_h^\ast~|~{\rm meas}_{d-1}(T \cap \Omega) >0\,\}, \quad \Omega_h:=\cup_{T \in \T_h} T.
\end{equation}
The set of simplices in $\T_h$ that have nonzero intersection with $\Gamma$ and the corresponding ``boundary strip'' are defined by
\[
  \cTG:=\{\, T \in \T_h~|~T \cap \Gamma\neq \emptyset \\,\}, \quad \Omega_h^\Gamma:= \cup_{T \in \cTG} T.  
\]
Furthermore, we define the interior domain $\Omega_h^0:=\Omega_h\setminus \Omega_h^\Gamma \subset \Omega$. We refer to Figure~\ref{fig:sketch} for illustration.
To avoid technical difficulties we assume that $\cTG$ coincides with the set of simplices that have a nonzero intersection with $\partial \Omega_h$, i.e.,
\begin{equation} \label{ass1}
 \cTG = \{\, T\in \T_h~|~T \cap \partial \Omega_h \neq \emptyset\,\}.
\end{equation}
This assumption is satisfied if $\Gamma$ is sufficiently resolved by the triangulation $\T_h$, i.e.,  for $h$ sufficiently small. For the ghost penalty stabilization we need a subset of the faces $F$ of $T \in \cTG$:
\[
  \F_g:= \{\, F \subset \partial T~|~T \in \cTG,~~\F \not\subset \partial \Omega_h\,\}.
\]
For the discretization of \eqref{Poisson} we use, for a fixed polynomial degree $k \geq 1$, the standard $H^1$-conforming finite element space on the fictitious domain $\Omega_h$:
\begin{equation} \label{FEspace}
  V_h:=\{\, v_h \in C(\Omega_h)~|~ {v_h}_{|T} \in \Pol_k~~\text{for all}~T \in \T_h\,\}.
\end{equation}
On this space we define the stabilized Nitsche bilinear form
\begin{equation} \label{defah} \begin{split}
  a_h(u,v)& := (\nabla u, \nabla v)_\Omega - (n_\Gamma\cdot \nabla u, v)_\Gamma - (u, n_\Gamma\cdot \nabla v)_\Gamma +
   \gamma (h^{-1} u,v)_\Gamma  \\ & + \beta \sum_{\ell=1}^k \sum_{F \in \F_g} h_F^{2 \ell-1}(\jump{\partial_n^\ell u},\jump{\partial_n^\ell v})_{F},
   \end{split}
\end{equation}
where $n_\Gamma$ is the outward pointing unit normal on $\Gamma$, $(f,g)_\omega=\int_\omega f g\, dx$  the $L^2$ scalar product on $\omega$, $\jump{\cdot}$ the usual jump operator (across the face $F$) and $\partial_n$ the derivative in direction normal to the face $F$. Since we do not assume quasi-uniformity of the triangulation, the scaling with $h^{-1}$ is element-wise, i.e., $(h^{-1}u,v)_{\Gamma}:=\sum_{T\in \cTG} h_T^{-1}(u,v)_{T\cap \Gamma}$.
The parameters $\gamma >0$, $\beta >0$ are fixed. The bilinear form in the second line of \eqref{defah} is the ghost penalty stabilization. Different equivalent variants of this stabilization are known in the literature, cf. \cite{Burman2010,Preuss2018,LehrenfeldOlshanskii2019}. The  choice of a particular variant of this stabilization is not relevant for the analysis in this paper. The discrete problem is as follows: determine $u_h \in V_h$ such that
\begin{equation} \label{discrete}
 a_h(u_h,v_h)= (f,v_h)_\Omega + \gamma(h^{-1} g, v_h)_\Gamma - (g,n_\Gamma\cdot \nabla v_h)_\Gamma\quad \text{for all}~v_h \in V_h.
\end{equation}
Note that for the implementation of this method one needs sufficiently accurate quadrature on cut simplices $T \cap \Omega$ and on the boundary segments $T \cap \Gamma$. For the case $k=1$ sufficient accuracy is  obtained by piecewise linear approximations of the boundary. For higher order $k \geq 2$ the efficient numerical realization of a sufficiently accurate  quadrature is not straightforward and variants have been developed that lead to optimal order bounds \cite{burmanhansboBVcorrection,lehrenfeldfictdomain}. For the analysis of this paper this ``geometric error'' does not play an essential role. Therefore we assume that the \emph{integrals in} \eqref{defah} \emph{are determined exactly}. In the literature \emph{optimal order discretization bounds} for the method \eqref{discrete} have been derived \cite{burmanhansbo12}, provided $\gamma > 0$ is taken sufficiently large.  In the remainder we assume that the latter holds.
\begin{remark}\label{rem1} \rm
 We briefly discuss the key ingredients needed for  the derivation of optimal discretization error bounds.
 Firstly, the following two trace inequalities are used:
 \begin{align}
   \|v\|_{\partial T} & \leq c (h_T^{-\frac12} \|v\|_T + h_T^\frac12 \|\nabla v\|_T), \quad v \in H^1(T), \label{trace1} \\ \|v\|_{T \cap \Gamma} & \leq c (h_T^{-\frac12} \|v\|_T + h_T^\frac12 \|\nabla v\|_T), \quad v \in H^1(T). \label{trace2}
 \end{align}
The constant $c$ in \eqref{trace2} is independent of how $\Gamma$ cuts the simplex $T \in \cTG$. Related to the effect of the ghost penalty stabilization we recall the following result \cite{Massing2014}[Lemma 5.1]. Let $T_1,T_2 \in \T_h$ be two simplices sharing a common face $F$. Let $v$ be a piecewise polynomial function relative to the macro-element $T_1 \cup T_2$ and let $v_i$ denote the restriction of $v$ to $T_i$, $i=1,2$. Then there is a constant $C$ depending only on the shape regularity of $\T_h$ and on $p:=\max\{{\rm deg}(v_1),{\rm deg}(v_2)\}$ such that
\begin{equation} \label{ghostestimate}
  \|v\|_{T_1}^2 \leq C\big(\|v\|_{T_2}^2 + \sum_{j \leq p} h_F^{2 j+1} \|\jump{\partial_n^j v}^2\|_F^2 \big).
\end{equation}
Furthermore, in the derivation of discretization error bounds one uses standard finite element inverse inequalities and optimal interpolation error estimates for the (standard) finite element space $V_h$. 
\end{remark}

In the discretization error analysis one uses a norm induced by the bilinear form
\begin{equation} \label{defb}
  b_h(u,v):=(\nabla u, \nabla v)_{\Omega_h} + \gamma (h^{-1} u,v)_\Gamma,
\end{equation}
which defines a scalar product on $H^1(\Omega_h)$ that is equivalent to the standard scalar product on this space. 
Note that in \eqref{defb} we use an integral over $\Omega_h$, whereas in the definition of $a_h(\cdot,\cdot)$ in \eqref{defah} an integral over $\Omega$ occurs.
The bilinear form $b_h(\cdot,\cdot)$ defines a norm on $V_h$ denoted by $\|\cdot\|_b$. A key result in the discretization error analysis \cite{burmanhansbo12,Massing2014}, which is derived using \eqref{trace1}, \eqref{trace2} and \eqref{ghostestimate}, is the following:
\[
 \|v_h\|_b^2 \lesssim a_h(v_h,v_h) \lesssim \|v_h\|_b^2, \quad \text{for all}~ v_h \in V_h. 
 \]
Here and in the remainder of the paper we use the standard notation: the constant in $\lesssim$ is independent of $h$ and of how $\Gamma$ cuts the triangulation. We use the notation $\|v_h\|_a:=a_h(v_h,v_h)^\frac12$. Hence, we have the \emph{uniform norm equivalence}
\begin{equation} \label{normeq1}
 \|v_h\|_b  \sim \|v_h\|_a \quad \text{for all}~ v_h \in V_h. 
 \end{equation} 
On $V_h$ we use the nodal finite element basis, denoted by $(\phi_i)_{1 \leq i \leq N}$. For $u_h, v_h \in V_h$ we have representations $u_h=\sum_{i=1}^N u_i \phi_i$, $v_h=\sum_{i=1}^N v_i \phi_i$. We introduce the notation $\bu=(u_1,\ldots,u_N)^T$, $\bv:=(v_1,\ldots,v_N)^T$. The discrete problem \eqref{discrete} leads to a linear system of the form
\begin{equation} \label{linsystem}
 \bA \bu = \mathbf{b}, \quad \text{with}~~ \bA_{ij}=a_h(\phi_i,\phi_j), \quad 1 \leq i,j \leq N.
\end{equation}
In the remainder of the paper we introduce and analyze an optimal preconditioner for the stiffness matrix $\bA$. This preconditioner is based on a subspace splitting that we study in the next section.

\section{Stable subspace splitting} \label{sectsubspace}
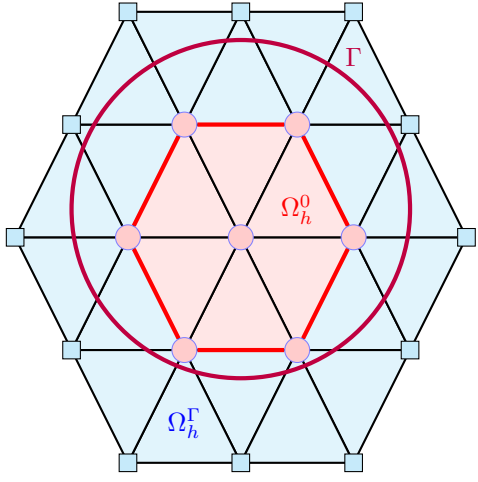
\begin{figure}
  \newcommand{\triang}[4]{\filldraw[#4] (#1.center) -- (#2.center) -- (#3.center) -- cycle; }
  \begin{center}
  \begin{tikzpicture}
    [scale=0.75,
     inner vert/.style={circle,draw=blue!50,fill=red!20},
     bnd vert/.style={rectangle,draw=black,fill=cyan!20},
     inner triang/.style={thick,draw=black,fill=red!10},
     bnd triang/.style={thick,draw=black,fill=cyan!10}
    ]
    \node (e1) at (-2,4) [bnd vert] {};
    \node (e2) at (0,4) [bnd vert] {};
    \node (e3) at (2,4) [bnd vert] {};
    \node (d1) at (-3,2) [bnd vert] {};
    \node (d2) at (-1,2) [inner vert] {};
    \node (d3) at (1,2) [inner vert] {};
    \node (d4) at (3,2) [bnd vert] {};
    \node (c1) at (-4,0) [bnd vert] {};
    \node (c2) at (-2,0) [inner vert] {};
    \node (c3) at (0,0) [inner vert] {};
    \node (c4) at (2,0) [inner vert] {};
    \node (c5) at (4,0) [bnd vert] {};
    \node (b1) at (-3,-2) [bnd vert] {};
    \node (b2) at (-1,-2) [inner vert] {};
    \node (b3) at (1,-2) [inner vert] {};
    \node (b4) at (3,-2) [bnd vert] {};
    \node (a1) at (-2,-4) [bnd vert] {};
    \node (a2) at (0,-4) [bnd vert] {};
    \node (a3) at (2,-4) [bnd vert] {};

    \draw[ultra thick,purple] (0,0.5) circle [radius=3]
       ++(2,2.7) node[purple]  {$\Gamma$};
    
    \begin{scope}[on background layer]
      \def\innernodes{b3, b2, c2, d2, d3, c4};
      \foreach \n [remember=\n as \nlast (initially c4)] in \innernodes {
        \triang{\nlast}{\n}{c3}{inner triang}      
      }
      \def\bndnodes{b3/b4/a3, b2/a2/a1, c2/b1/c1, d2/d1/e1, d3/e2/e3, c4/d4/c5};
      \foreach \n / \bn / \bbn 
        [remember=\n as \nlast (initially c4)
        ,remember=\bbn as \bbnlast (initially c5)] in \bndnodes 
      {
          \triang{\n}{\bn}{\nlast}{bnd triang}   
          \triang{\n}{\bn}{\bbn}  {bnd triang}      
          \triang{\nlast}{\bn}{\bbnlast}  {bnd triang}      
      }
      \foreach \n [remember=\n as \nlast (initially c4)] in \innernodes
        \draw[red,ultra thick] (\n) -- (\nlast);
    \end{scope}
    \node[red] at (1,0.5) {$\Omega_h^0$};
    \node[blue] at (-1,-3.3) {$\Omega_h^\Gamma$};
  \end{tikzpicture}
  \end{center}
  \caption{Sketch of interface $\Gamma$ and triangulation $\T_h$ of $\Omega_h$ with interior nodes $\xi_i$, $1\leq i \leq N_0$ (red circles) and boundary nodes $\xi_i$, $N_0+1\leq i \leq N$ (blue squares). The blue triangles form the triangulation $\T_h^\Gamma$ of $\Omega_h^\Gamma$, the red triangles form the triangulation of $\Omega_h^0$. The edges between the red and blue triangles form the interior boundary $\partial\Omega_h^0$.}
  \label{fig:sketch}
\end{figure}

In this section we introduce a  natural splitting of the space $V_h$. The finite element nodes corresponding to the nodal basis functions $\phi_i$ are denoted by $\xi_i$, $1 \leq i \leq N$. We choose the ordering such that $\xi_i$, $1 \leq i \leq N_0,$ are all nodes in the interior of $\Omega_h$ and $\xi_i$, $N_0+ 1 \leq i \leq N$, are the boundary nodes, i.e., $\xi_i \in \partial \Omega_h$ iff $i > N_0$, cf. Figure~\ref{fig:sketch}. The corresponding splitting is
\begin{equation} \label{splitting} 
 \begin{split}
    V_h & = V_h^0 \oplus V_h^\Gamma\\
    V_h^0 & := {\rm span}\{\, \phi_i~|~ 1 \leq i \leq N_0\,\} = \{ \, v_h \in V_h~|~{v_h}_{|\partial \Omega_h}=0\,\} \\
    V_h^\Gamma & := {\rm span}\{\, \phi_i~|~ N_0+1 \leq i \leq N\,\}.
 \end{split}
\end{equation}
Below we show that this splitting is stable (Theorem~\ref{thmmain1}). This is essentially based on the following elementary observation. We split the boundary of the boundary strip $\Omega_h^\Gamma$ into two disjoint parts, $\partial \OG = \partial \Omega_h \cup \partial \Omega_h^0$, with $\partial \Omega_{h}^0 \subset \Omega$, cf. Figure~\ref{fig:sketch}. Note that the following holds:
\begin{equation} \label{elprop}
 \begin{split}
  v_h \in V_h^0 ~&\Rightarrow~~v_h=0 \quad \text{on}~\partial\Omega_h \\
  v_h \in V_h^\Gamma~& \Rightarrow ~~v_h=0 \quad \text{on}~ \partial \Omega_{h}^0.
 \end{split}
\end{equation}
\begin{lemma} \label{lemma1}
 The following uniform norm equivalence holds:
 \begin{equation} \label{normeq}
  \|h^{-1} v_h\|_{\OG} \sim \|\nabla v_h\|_{\OG} \quad \text{for all}~v_h \in V_h^0 \cup V_h^\Gamma.
 \end{equation}
\end{lemma}
\begin{proof}
 The estimate in the one direction directly follows from a standard finite element inverse inequality:
 \[
   \|\nabla v_h\|_{\OG}^2 = \sum_{T\in\cTG}\|\nabla v_h\|_{T}^2 \lesssim  \sum_{T\in\cTG} h_T^{-2}\| v_h\|_{T}^2 = \|h^{-1} v_h\|_{\OG}^2 \quad \text{for all}~v_h \in V_h.
 \]
Take $v_h \in V_h^0 \cup V_h^\Gamma$ and $T \in \cTG$. By construction $T$ has at least one vertex on $\partial \Omega_h$ and at least one  vertex on $\partial \Omega_{h}^0$. Using this and \eqref{elprop} it follows that there is vertex of $T$, denoted by $x_\ast$, at which $v_h(x_\ast)=0$ holds.
Let $\hat T$ be the unit simplex and $ F: \hat T \to T$ the affine transformation with $F(0)=x_\ast$. Define $Z:=\{\, p \in \mathcal{P}_k~|~p(0)=0\,\}$ and note that $p \to \|p\|_{\hat T}$ and $p \to \|\nabla p\|_{\hat T}$ define equivalent norms on $Z$. Due to $\hat v_h:=v_h \circ F \in Z$ and this norm equivalence we obtain
\[
  \|v_h\|_T^2 = |T| \|\hat v_h\|_{\hat T}^2 \lesssim |T| \|\nabla \hat v_h\|_{\hat T}^2 \lesssim h_T^2 \|\nabla v_h\|_T^2,
\]
and thus
\[
 \|h^{-1} v_h\|_{\OG}^2 = \sum_{T \in \cTG} h_T^{-2} \|v_h\|_T^2 \lesssim \sum_{T\in \cTG} \|\nabla v_h\|_T^2=\|\nabla v_h\|_{\OG}^2,
\]
which is this estimate in the other direction.
\end{proof}
\ \\
\begin{corollary} \label{corolmain}
 The following uniform norm equivalences hold:
 \begin{align} \label{normeq2}
 \|v_h\|_b &\sim \|\nabla v_h\|_{\OG} \sim \|h^{-1} v_h\|_{\OG} \quad \text{for all}~v_h \in  V_h^\Gamma \\
 \|v_h\|_b &\sim \|\nabla v_h\|_{\Omega_h}\quad \text{for all}~v_h \in  V_h^0.\label{normeq3}
 \end{align}
\end{corollary}
\begin{proof}
 The second norm equivalence  in \eqref{normeq2} is the result of Lemma~\ref{lemma1}.
 For $ v_h \in V_h^\Gamma$ we have $v_h=0$ on $\Omega_h^0$, hence, $\|v_h\|_b^2= \|\nabla v_h\|_{\Omega_h}^2 + \|h^{-\frac12} v_h\|_\Gamma^2=\|\nabla v_h\|_{\OG}^2 + \|h^{-\frac12} v_h\|_\Gamma^2$. Using this, the estimate $ \|\nabla v_h\|_{\OG} \leq \|v_h\|_b$ is trivial. For the estimate in the other direction we use \eqref{trace2}, a standard finite element inverse estimate and Lemma~\ref{lemma1}:
 \begin{equation} \label{hest1} \begin{split}
  \|h^{-\frac12} v_h\|_\Gamma^2  & = \sum_{T \in \cTG} h_T^{-1} \|v_h\|_{T\cap \Gamma}^2 \lesssim \sum_{T \in \cTG} h_T^{-2} \|v_h\|_T^2 + \|\nabla v_h\|_T^2  \\ & \lesssim
  \sum_{T \in \cTG} h_T^{-2}\|v_h\|_T^2 = \|h^{-1} v_h\|_{\OG}^2 \lesssim \|\nabla v_h\|_{\OG}^2. 
 \end{split} \end{equation}
 Combining this proves the first norm equivalence in \eqref{normeq2}. We now consider \eqref{normeq3}. The estimate in one direction is trivial. Note that the result \eqref{hest1} holds also for $v_h \in V_h^0$. Using this we obtain the estimate in the other direction in \eqref{normeq3}.
\end{proof}
\ \\[1ex]
From the result in Lemma~\ref{lemma1} it follows that for $v_h \in (V_h^0 \cup V_h^\Gamma) \subset V_h$ the scaled $L^2$-norm $\|h^{-1} v_h\|_{\OG}$ is uniformly bounded by $\|\nabla v_h\|_{\OG}$. We need one further result in which for arbitrary $v_h \in V_h$ a suitable bound of this scaled $L^2$-norm is derived.
This result is very similar to results known from the literature, cf. Remark~\ref{Remsimilar}. We introduce, for $T \in \cTG$, the subdomain consisting of all simplices  in $\cTG$ that have at least a common vertex with $T$, i.e., $\omega_T:= \{\, \tilde T \in \cTG~|~\tilde T \cap T \neq \emptyset \,\}$. Note that due to shape regularity we have $h_{\tilde T} \sim h_T$ for $\tilde T \in \omega_T$ and ${\rm diam}(\omega_T) \sim h_T$. 
\begin{lemma} \label{lemma2}
For arbitrary $T \in \cTG$ the following holds:
\begin{equation}
 h_T^{-2} \|v_h\|_T^2 \lesssim \sum_{\tilde T \in \omega_T} h_{\tilde T}^{-1} \|v_h\|_{\tilde T \cap \Gamma}^2 + \|\nabla v_h\|_{\omega_T}^2 \quad \text{for all}~v_h \in V_h.
\end{equation}
 The constant in $\lesssim$ depends only on shape regularity of $\T_h$, smoothness of $\Gamma$ and the polynomial degree $k$ used in $V_h$.
\end{lemma}
\begin{proof}
 Take $T \in \cTG$, $v_h \in V_h$. The area $|T\cap \Gamma |$ can be arbitrary small (``small cuts''), but it follows from \cite[Proposition 4.2]{DemOlsh} that there is an element $\tilde T \in \omega_T$ such that $|\tilde T \cap \Gamma| \geq c_0 h_{\tilde T}^{d-1}$, with  a constant $c_0 >0$ that depends only on shape regularity of $\cTG$ and on smoothness of $\Gamma$. Take such an $\tilde T \in \omega_T$. Take a fixed  $\xi \in  \Gamma \cap  \tilde T$ such  that $|v_h(\xi)|= \max_{x \in \tilde T \cap \Gamma} |v_h(x)| =: \|v_h\|_{\infty,\tilde T \cap \Gamma}$. Take $x \in T $ and let $S$ be a smooth shortest curve in $\omega_T$ that connects $x$ and $\xi$. Due to shape regularity we have $|S| \lesssim h_T$, independent of $x$. This yields
 \[
  v_h(x)= v_h(\xi) + \int_S\frac{\partial v_h}{\partial s} \, ds,
 \]
 with $s$ the arclength parametrization of $S$.
Hence,
\[
  v_h(x)^2 \leq  2 v_h(\xi)^2 + 2 |S|^2 \|\nabla v_h\|_{\infty,\omega_T}^2.
\]
Using integration over $T$, $|T|\sim h_T^d$ and the standard FE norm estimate $\|\nabla v_h\|_{\infty,\omega_T}^2 \lesssim h_T^{-d} \|\nabla v_h\|_{\omega_T}^2$ we get
\begin{equation} \label{est6}
 h_T^{-2} \|v_h\|_T^2 \lesssim h_T^{d-2} \|v_h\|_{\infty,\tilde T \cap \Gamma}^2 + \|\nabla v_h\|_{\omega_T}^2.
\end{equation}
Using $|\tilde T \cap \Gamma| \geq c_0 h_{\tilde T}^{d-1}$ we get
\[
  \|v_h\|_{\infty,\tilde T \cap \Gamma}^2 \lesssim h_{\tilde T}^{1-d} \|v_h\|_{\tilde T \cap \Gamma}^2,
\]
and combining this with the result \eqref{est6} and $h_{\tilde T} \sim h_T$ completes the proof.
\end{proof}
\ \\
By summing over $T \in \cTG$ and using a standard finite overlap argument we obtain the following result.
\begin{corollary} \label{corolA}
 The following uniform estimate holds:
 \begin{equation} \label{estfund}
   \|h^{-1} v_h \|_{\OG} \lesssim \|h^{-\frac12} v_h\|_\Gamma + \|\nabla v_h\|_{\OG} \quad \text{for all}~v_h \in V_h.
 \end{equation}
\end{corollary}
\ \\
\begin{remark}\label{Remsimilar}
 \rm As noted above, similar results are known in the literature. For example, in the papers \cite{burmanembedded,grande2017higher}, for the case of a \emph{quasi-uniform} triangulation the following uniform estimate is derived:
 \begin{equation} \label{compresult}
  \|v_h\|_{\OG} \lesssim h^\frac12 \|v_h\|_{\Gamma} + h \|n \cdot \nabla v_h\|_{\OG}.
 \end{equation}
Note that due to the quasi-uniformity assumption we have a simpler scaling with the global mesh parameter $h$ and that in \eqref{compresult} we  have the \emph{normal} derivative term $\|n \cdot \nabla v_h\|_{\OG}$, with $n$ the normal on $\Gamma$ (constantly extended in the neighborhood $\OG$) instead of the full derivative term $\|\nabla v_h\|_{\OG}$. The proofs of \eqref{compresult} in \cite{burmanembedded,grande2017higher}  are much more involved than the simple proof of Lemma~\ref{lemma2} above. This is due to the fact that in the bound in \eqref{compresult} only the normal derivative occurs.
\end{remark}
\ \\
We need the following elementary estimate.
\begin{lemma} \label{lemma3}
 Let $M \in \mathbb{R}^{m\times m}$ be symmetric positive definite and $\kappa(M):=\|M\|_2\|M^{-1}\|_2$ the spectral condition number. For all $\bx, \by \in \mathbb{R}^m$ with $\langle \bx, \by\rangle = \bx^T \by=0$ the following holds:
 \[
   |\langle M \bx, \by \rangle | \leq \Big(1- \frac{1}{\kappa(M)}\Big) \langle M \bx, \bx \rangle^\frac12 \langle M \by, \by \rangle^\frac12.
 \]
 \end{lemma}
\begin{proof}
 Let $MV=V\Lambda$, with $\Lambda={\rm diag}(\lambda_1, \ldots, \lambda_m)$, $0<\lambda_1 \leq \ldots \leq \lambda_m$, $V^TV=I$ be the eigenvector decomposition of $M$. Take $\bx, \by \in \mathbb{R}^m$ with $\langle \bx, \by\rangle=0$ and define $\hat \bx :=V^T \bx$, $\hat \by = V^T \by$. This yields $\langle \hat \bx, \hat \by \rangle =0$, i.e., $\hat x_1 \hat y_1= - \sum_{i=2}^m \hat x_i \hat y_i$. Using this we obtain
 \begin{align*}
 |\langle M \bx, \by \rangle | & =|\langle \Lambda \hat \bx, \hat \by \rangle | = |\sum_{i=1}^m\lambda_i \hat x_i \hat y_i | \\ 
 & =|\sum_{i=2}^m (\lambda_i-\lambda_1)\hat x_i \hat y_i | \leq \max_{2 \leq i \leq m} \frac{\lambda_i-\lambda_1}{\lambda_i} \sum_{i=2}^m \lambda_i |\hat x_i|| \hat y_i |\\
 & \leq \Big( 1- \frac{\lambda_1}{\lambda_m}\Big) \Big(\sum_{i=1}^m \lambda_i \hat x_i^2 \Big)^\frac12 
\Big(\sum_{i=1}^m \lambda_i \hat y_i^2 \Big)^\frac12 \\
 & = \Big( 1- \frac{1}{\kappa(M)}\Big) \langle M \bx, \bx \rangle^\frac12 \langle M \by, \by \rangle^\frac12, 
 \end{align*}
which proves the result.
\end{proof}
\ \\
Using this we obtain the following uniform strengthened Cauchy-Schwarz inequality.
\begin{lemma} \label{lemma4}
 Let $\hat M \in \mathbb{R}^{m \times m}$, $m:= \begin{pmatrix} d+ k \\ k\end{pmatrix}$, be the element mass matrix on the reference unit simplex $\hat T \subset \mathbb{R}^d$. For $ T\in \cTG$ the estimate
 \begin{equation} \label{est7}
 |(v_h^0,v_h^\Gamma)_T| \leq \Big(1- \frac{1}{\kappa(\hat M)}\Big) \|v_h^0\|_T \|v_h^\Gamma\|_T \quad \text{for all}~v_h^0 \in V_h^0, \, v_h^\Gamma \in V_h^\Gamma,
\end{equation}
holds.
\end{lemma}
\begin{proof}
 Take $T\in \cTG$, $v_h^0 \in V_h^0, \, v_h^\Gamma \in V_h^\Gamma$, $v_h:=v_h^0 + v_h^\Gamma$. On $T$ we introduce a local numbering of the element nodal basis functions such that $\phi_i$, $1 \leq i \leq m_0$, correspond to nodes $\xi_i$ in the interior of $\Omega_h$. Note that $T$ has at least one vertex on $\partial \Omega_h$ and thus $1 \leq m_0 < m$ holds. By construction we have
 \[
   v_h^\Gamma(\xi_i)=0~~\text{for}~i=1,\ldots, m_0, \quad v_h^0(\xi_i)=0~~\text{for}~i=m_0+1,\ldots, m.
 \]
The representation in the local basis has a splitting 
\[ {v_h}_{|T}= \sum_{i=1}^m \alpha_i \phi_i = \sum_{i=1}^{m_0} \alpha_i \phi_i+ \sum_{i=m_0+1}^m \alpha_i \phi_i={v_h^0}_{|T} +{v_h^\Gamma}_{|T}.
\]
The corresponding coefficient vector splitting of $\boldsymbol{\alpha}=(\alpha_1, \ldots, \alpha_m)^T$ is 
\[ \boldsymbol{\alpha}=(\alpha_1, \ldots, \alpha_{m_0},0,\ldots, 0)^T+ (0, \ldots, 0,\alpha_{m_0+1}, \ldots, \alpha_m)^T =: \boldsymbol{\alpha}^0 +\boldsymbol{\alpha}^\Gamma.\]
Note that $\langle \boldsymbol{\alpha}^0,\boldsymbol{\alpha}^\Gamma \rangle =0$ holds. Let $M \in \mathbb{R}^{m \times m}$, $M_{i,j}= (\phi_i, \phi_j )_T $, be the element mass matrix. Note that $\kappa (M)=\kappa(\hat M)$ holds. Thus we obtain, using Lemma~\ref{lemma3}:
\begin{align*}
 |(v_h^0,v_h^\Gamma)_T| & = |\langle M\boldsymbol{\alpha}^0,\boldsymbol{\alpha}^\Gamma \rangle| \leq
   \Big(1- \frac{1}{\kappa(M)}\Big) \langle M\boldsymbol{\alpha}^0,\boldsymbol{\alpha}^0 \rangle^\frac12 
    \langle M\boldsymbol{\alpha}^\Gamma,\boldsymbol{\alpha}^\Gamma \rangle^\frac12 \\
    & = \Big(1- \frac{1}{\kappa (\hat M)}\Big) \|v_h^0\|_T \|v_h^\Gamma\|_T,
\end{align*}
which completes the proof.
\end{proof}

Based on these lemmata we derive the following stable splitting main result.
\begin{theorem} \label{thmmain1} We decompose $v_h \in V_h$ as $v_h= v_h^0 +v_h^\Gamma$, $v_h^0 \in V_h^0$, $v_h^\Gamma \in V_h^\Gamma$.  The following holds with constants $K_b$, $K_a$ independent of $u_h$, of $h$ and of how $\Gamma$ intersects $\OG$:
\begin{align}
  \|v_h^0\|_b^2 +\|v_h^\Gamma\|_b^2  & \leq K_b \|v_h\|_b^2 \label{mainb}\\
  \|v_h^0\|_a^2 +\|v_h^\Gamma\|_a^2  & \leq K_a \|v_h\|_a^2 .\label{maina}
\end{align}
\end{theorem}
\begin{proof}
 The result in \eqref{maina} is a direct consequence of \eqref{mainb} and \eqref{normeq1}. We derive the result \eqref{mainb} as follows. Note that $v_h^\Gamma=0$ on $\Omega_h^0$, i.e., $v_h=v_h^0$ on  $\Omega_h^0$. Using this, Corollary~\ref{corolmain} and a finite element inverse estimate we get
 \begin{align}
  \|v_h^0\|_b^2 +\|v_h^\Gamma\|_b^2  & \sim \|\nabla v_h^0\|_{\Omega_h}^2 +\|\nabla v_h^\Gamma\|_{\OG}^2 \nonumber \\
   &  = \|\nabla v_h\|_{\Omega_h^0}^2 +\|\nabla v_h^0\|_{\OG}^2 +\|\nabla v_h^\Gamma\|_{\OG}^2 \nonumber \\
   & \lesssim \|\nabla v_h\|_{\Omega_h^0}^2 + \|h^{-1} v_h^0\|_{\OG}^2 + \|h^{-1} v_h^\Gamma\|_{\OG}^2. \label{est8}
 \end{align}
The result in Lemma~\ref{lemma4} yields $\|v_h^0\|_T^2 + \|v_h^\Gamma\|_T^2 \leq \kappa(\hat M) \|v_h\|_T^2$ for all $T\in \cTG$. Thus we get
\begin{align*}
 \|h^{-1} v_h^0\|_{\OG}^2 +\| h^{-1}v_h^\Gamma\|_{\OG}^2 &  = \sum_{T \in \cTG} h_T^{-2} \big( \|v_h^0\|_T^2 + \|v_h^\Gamma\|_T^2\big) \\ &  \leq \kappa(\hat M) \sum_{T \in \cTG} h_T^{-2} \|v_h\|_T^2 =\kappa(\hat M) \|h^{-1}v_h\|_{\OG}^2. 
\end{align*}
Using this in \eqref{est8} and applying Corollary~\ref{corolA} we get
\[
 \|v_h^0\|_b^2 +\|v_h^\Gamma\|_b^2 \lesssim \|\nabla v_h\|_{\Omega_h^0}^2 +  \|h^{-1}v_h\|_{\OG}^2 \lesssim
  \|\nabla v_h\|_{\Omega_h}^2 + \|h^{-\frac12} v_h\|_\Gamma^2 \sim \|v_h\|_b^2,  
\]
which completes the proof.
\end{proof}

\section{An optimal preconditioner} \label{sectPrecond}

 We introduce and analyze an additive subspace decomposition preconditioner using the framework given in \cite{Yserentant:93}. 
  It is convenient to introduce the notation $V_h^1:=V_h^\Gamma$, i.e. we have the stable splitting $V_h=V_h^0\oplus V_h^1$. Let $Q_l: V_h \to V_h^l$, $l=0,1$,  be the $L^2$-projection, i.e., for $u \in V_h$:
  \[
  (Q_l u, w_l)_{\Omega_h}= (u,w_l)_{\Omega_h} \quad \text{for all}~w_l \in V_h^l.
  \]
  The bilinear form $a_h(\cdot,\cdot)$ on $V_h$ that defines the fictitious domain discretization \eqref{defah} can be represented by the operator $A:\, V_h \to V_h$:
  \begin{equation} \label{defA}
  (Au,v)_{\Omega_h} =a_h(u,v) \quad \text{for all}~u,v \in V_h.
  \end{equation}
  The discrete problem \eqref{discrete} has the compact representation $A u=f_Q$, where $f_Q \in V_h$ is the  Riesz representation of the right-hand side functional on $V_h$.  The Ritz approximations $A_l: V_h^l \to V_h^l$, $l=0,1$, of $A$ are given by
  \[
  (A_l u,v)_{\Omega_h}=(Au,v)_{\Omega_h}=a_h(u,v) \quad \text{for all}~ u,v \in V_h^l.
  \]
  Note that these are symmetric positive definite operators.
  In the preconditioner we need symmetric positive definite approximations $B_l : V_h^l \to V_h^l$ of the Ritz operators $A_l$. The spectral equivalence of $B_l$ and $A_l$ is described by the following:
  \begin{equation} \label{spectral}
    \gamma_l ( B_l u,u)_{\Omega_h}\leq (A_l u,u)_{\Omega_h} \leq \rho_l( B_l u,u)_{\Omega_h}  \quad \text{for all}~u \in V_h^l,
  \end{equation}
  with strictly positive constants $\gamma_l$, $\rho_l$, $l=0,1$. The \emph{additive subspace preconditioner} is defined by
  \begin{equation} \label{additive}
    C=B_0^{-1} Q_0+B_1^{-1} Q_1.
  \end{equation}
  For the implementation of this preconditioner one has to solve (in parallel) two linear systems. The operator $Q_l$ is not  needed in the implementation, since if for a given $z \in V_h$ one has to determine $d_l =B_l^{-1} Q_l z$, the solution can be obtained as follows: determine $d_l \in V_h^l$ such that
  \[
  (B_l d_l, v)_{\Omega_h}=(z,v)_{\Omega_h}  \quad \text{for all}~v \in V_h^l.
  \]
  The theory presented in \cite{Yserentant:93} can be used to quantify the quality of the preconditioner $C$.
  \begin{theorem} \label{MMain}
    Define $\gamma_{\min}= \min_l \gamma_l$, $\rho_{\max} = \max_l \rho_l$. Let $K_a$ be the constant of the stable splitting in \eqref{maina}. The spectrum $\sigma (CA)$ is real and
    \[
    \sigma(CA) \subset \big[ \frac{\gamma_{\min}}{K_a }, 2 \rho_{\max} \big]
    \]
    holds.
  \end{theorem}
  \begin{proof}
    We recall a main result from \cite[Theorem 8.1]{Yserentant:93}. If there are strictly positive constants $K_1, K_2$ such that
    \[
    K_1^{-1} \sum_{l=0}^1 (B_l u_l,u_l)_{\Omega_h} \leq \|u_0+u_1\|_a^2 \leq K_2 \sum_{l=0}^1 (B_l u_l,u_l)_{\Omega_h}\quad \text{for all}~u_l \in V_h^l,
    \]
    is satisfied, then $\sigma(CA) \subset [K_1^{-1},K_2]$ holds.
    For the lower bound we use Theorem~\ref{thmmain1} and \eqref{spectral}, which then results in
    \[
    \|u_0+u_1\|_a^2 \geq K_a^{-1} \sum_{l=0}^1\|u_l\|_a^2 = K_a^{-1} \sum_{l=0}^1 (A_l u_l,u_l)_{\Omega_h} \geq \frac{\gamma_{\min}}{K_a}  \sum_{l=0}^1 (B_l u_l,u_l)_{\Omega_h}.
    \]
    For the upper bound we note
    \[
    \|u_0+u_1\|_a^2  \leq 2 \sum_{l=0}^1\|u_l\|_a^2 = 2 \sum_{l=0}^1 (A_l u_l,u_l)_{\Omega_h} \leq 2 \rho_{\max}\sum_{l=0}^1 (B_l u_l,u_l)_{\Omega_h}. 
    \]
    Now we apply the above-mentioned result with $K_1=K_a/\gamma_{\min}$ and $K_2=2 \rho_{\max}$.
  \end{proof}
  \ \\[1ex]
  The result in Theorem~\ref{thmmain1}  yields that the constant $K_a$ is independent of $h$ and of how the triangulation intersects the interface $\Gamma$. It remains to choose appropriate operators $B_l$ such that $\gamma_{\rm min}$ and $\rho_{\max}$ are uniform constants, too.

  We first consider the approximation $B_0$ of the Ritz-projection $A_0$ in $V_h^0$. Using \eqref{normeq1} and \eqref{normeq3} we get
  \[
  (A_0 u,u)_{\Omega_h}=a_h(u,u) \sim (\nabla u, \nabla u)_{\Omega_h}  \quad \text{for all}~~u \in V_h^0.
  \]
  Hence, $A_0$ is \emph{uniformly equivalent to a standard finite element discretization (in $V_h^0$) of the Poisson equation with zero Dirichlet boundary conditions on $\Omega_h$}. As a preconditioner $B_0$ for $A_0$ we can use a  symmetric multigrid method (which is a multiplicative subspace correction  method). There is a technical issue related to the nesting of spaces because the domain $\Omega_h$ varies with $h$. This is addressed in Remark~\ref{RemMG} below. One can also use an algebraic multigrid preconditioner applied to $A_0$, which is what we do in our numerical experiments, cf. section~\ref{sectNumExp}. For these choices of $B_0$ we have spectral inequalities as in \eqref{spectral}, with constants $\gamma_0 >0, \rho_0$ that are independent of $h$ and of how $\Gamma$ intersects the triangulation. 

  It remains to find an appropriate preconditioner $B_1$ of $A_1$, which is the Ritz projection in $V_h^1=V_h^\Gamma$. The norm equivalence \eqref{normeq2} implies that a simple diagonal scaling (Jacobi preconditioner) is already an optimal preconditioner. To derive  this result in the operator framework used above,  we  introduce the operator $B_1$ that represents the Jacobi preconditioner. Recall that $V_h^1= {\rm span}\{ \phi_i~|~ N_0+1 \leq i \leq N\, \} $. Define the index set $\I:=\{\, i~|~N_0+1 \leq i \leq N\, \}$.  Elements $u,v \in V_h^1$ have unique representations  $u= \sum_{i \in \I} \xi_i \phi_i$, $v=\sum_{i \in \I} \zeta_i \phi_i$.  In terms of these representations the Jacobi preconditioner is defined by
  \begin{equation} \label{Jacobi}
    (B_1 u,v)_{\Omega_h}= \sum_{i \in \I} \xi_i \zeta_i a_h( \phi_i,\phi_i), \quad u,v \in V_h^1.
  \end{equation}
  Note that $ a_h( \phi_i,\phi_i)$ are diagonal entries of the stiffness matrix $\MAT{A}$ in \eqref{linsystem}.  The result in the next lemma shows that this diagonal scaling  yields a robust preconditioner for the Ritz operator $A_1$. 
  \begin{lemma} \label{lemspectralJ}
For the Jacobi preconditioner $B_1$ there are strictly positive constants $\gamma_1$, $\rho_1$, independent of $h$ and of how the domain $\Omega_h$ is intersected by $\Gamma$ such that 
    \begin{equation} \label{spectralJ}
      \gamma_1 ( B_1 u,u)_{\Omega_h} \leq (A_1 u,u)_{\Omega_h} \leq \rho_1( B_1 u,u)_{\Omega_h}  \quad \text{for all}~u \in V_h^1
    \end{equation}
    holds.
  \end{lemma}
  \begin{proof} Take $u= \sum_{i \in \I} \xi_i \phi_i \in V_h^1$. For $T \in \cTG$ the set of finite element nodes in $T$ is denoted by $N(T)$. Standard arguments yield that $\|u\|_T^2 \sim |T| \sum_{i \in N(T)} \xi_i^2$ holds. Using this and the norm equivalences \eqref{normeq1}, \eqref{normeq2} we get
  \begin{equation} \label{eq3} \begin{split}
   (A_1u,u)_{\Omega_h} & =\|u\|_a^2 \sim \|u\|_b^2 \sim \|h^{-1} u\|_{\OG}^2  \\ & = \sum_{T \in \cTG} h_T^{-2} \|u\|_T^2 \sim
    \sum_{T \in \cTG} h_T^{-2} |T| \sum_{i \in N(T)} \xi_i^2. \end{split}
  \end{equation}
For the Jacobi operator we get, using \eqref{normeq1}, \eqref{normeq2}:
\begin{equation} \label{eq4} \begin{split}
 (B_1u,u)_{\Omega_h} & =\sum_{i \in \I} \xi_i^2 \|\phi_i\|_a^2 \sim \sum_{i \in \I} \xi_i^2 \|h^{-1}\phi_i\|_{\OG}^2 \sim \sum_{i \in \I}\xi_i^2 \sum_{T \in {\rm supp}(\phi_i)}  h_T^{-2} |T| \\
  & \sim \sum_{T \in \cTG}h_T^{-2} |T|\sum_{i \in N(T)} \xi_i^2 .
\end{split} \end{equation}
Comparing \eqref{eq3} and \eqref{eq4} it follows that the result \eqref{spectralJ} holds.
   \end{proof}
\ \\
\begin{remark} \label{RemMG}\rm
 We discuss a multigrid preconditioner for the subspace stiffness matrix given by
 \begin{equation} \label{matsub}
   \bA_{i,j}:= a_h(\phi_i,\phi_j), \quad 1 \leq i,j \leq N_0.
 \end{equation}
Recall that ${\rm span}\{\, \phi_i~|~1 \leq i \leq N_0\,\}= V_h^0= \{\, v_h \in V_h~|~{v_h}_{|\partial \Omega_h}=0\,\}$ and that the uniform norm equivalence $a_h(v,v)\sim (\nabla v,\nabla v)_{\Omega_h}$ for all $v \in V_h^0$ holds. For a multigrid preconditioner one needs a suitable hierarchy of ``coarser spaces'' and it is convenient if these are nested. In our setting a natural nested hierarchy is obtained as follows. We assume that on the larger polygonal domain $\Omega^\ast$ we have nested triangulations (coarse to fine) $\T_{h_0}^\ast, \T_{h_1}^\ast, \ldots ,\T_{h_{\J}}^\ast$, where $h_\J=h$ corresponds to the finest triangulation with stiffness matrix $\bA$ given in \eqref{linsystem}. This finest triangulation defines the fictitious  domain $\Omega_h$, cf. \eqref{Pa}. Corresponding to $\T_{h_j}^\ast$ we have standard nested finite element spaces $V_{h_j}^\ast:= \{\, v_h \in C(\Omega^\ast)~|~{v_h}_{|T} \in \mathcal{P}_k~~\text{for all}~T \in \T_{h_j}^\ast\,\}$. The nodal basis functions in the space $V_{h_j}^\ast$ are denoted by $\phi_i^{(j)}$, $i=1, \ldots {\rm dim}(V_{h_j}^\ast)$. We now define the hierarchy of nested subspaces:
\begin{equation} \label{defnested}
V_{h_j}^0:= {\rm span}\{\, \phi_i^{(j)}~|~ \phi_i^{(j)}(x)=0\quad \text{for all}~ x \notin \Omega_h\,\}, \quad j=0,1,\ldots, \J.
\end{equation}
Note that the domain $\Omega_h$ used in this definition is the fictitious domain on the finest level. One easily verifies that the nestedness property $V_{h_0}^0 \subset V_{h_1}^0 \subset \ldots V_{h_{\J}}^0= V_h^0$ holds. Thus we can construct a standard multigrid (multiplicative or additive) preconditioner. 
Optimal preconditioning properties of such a preconditioner for  the case of a level dependent (fictitious) domain $\Omega_h$ have been studied in the literature. For linear finite elements ($k=1$) and a quasi-uniform family of triangulations the analysis in \cite{Kornhuber} yields a condition number bound of the preconditioned matrix that grows linearly in the level number $\J$. The authors note that this very slow growth of the bound can be eliminated, i.e., one obtains an optimal preconditioner,  using similar ideas and techniques from \cite{Oswald}. They also note that all techniques and results can be generalized to the case of non-uniformly refined shape regular triangulations. For this one can work with $L^2$-like space decompositions that are based on local projections \cite{DahmenKunoth,Bornemann}. 
\end{remark}
\ \\

The analysis above leads to the following preconditioners for the linear system in \eqref{linsystem}
We use a matrix block partitioning according to the index set splitting in \eqref{splitting}.  The matrices corresponding to the Ritz approximations $A_l$ (projection on $V_h^l$)  are denoted by $\MAT{A}_l$, $l=0,1$.
Preconditioners of $\bA_l$  are denoted by  $\MAT{B}_l$, $l=0,1$, for instance $\MAT{B}_0$ a multigrid preconditioner and $\MAT{B}_1={\rm diag}(\bA_1)$ corresponding to the Jacobi method.  
  We define the block Jacobi preconditioners 
  \begin{equation} \label{preco}
    \MAT{P}_{\MAT{A}} := \left( \begin{array}{cc} \MAT{A}_0 & \MAT{0} \\ \MAT{0} & \MAT{A}_1  \end{array} \right), \quad
    \MAT{P}_{\MAT{D}} := \left( \begin{array}{cc} \MAT{A}_0 & \MAT{0} \\ \MAT{0} & \MAT{B}_1  \end{array} \right), \quad
    \MAT{P}_{\MAT{B}} := \left( \begin{array}{cc} \MAT{B}_0 & \MAT{0} \\ \MAT{0} & \MAT{B}_1 \end{array} \right).
  \end{equation}
The preconditioner $\MAT{P}_{\MAT{A}}$ corresponds to exact subspace solves and due to Theorem~\ref{MMain} we have the condition number bound $\kappa(\MAT{P}_{\MAT{A}}^{-1} \MAT{A}) \leq 2 K_a$. If in $\MAT{P}_{\MAT{D}}$ we use a Jacobi preconditioner for the block $ \MAT{B}_1$ then Lemma~\ref{lemspectralJ} yields
$\kappa(\MAT{P}_{\MAT{D}}^{-1} \bA) \leq 2 \frac{\rho_1}{\gamma_1} K_a$. Recall that the constants $K_a$, $\gamma_1$ and $\rho_1$ are independent of $h$ and of how the triangulation $\T_h$ intersects the boundary $\Gamma$. Finally, for ${\MAT{P}}_{\MAT{B}}$ we have $\kappa(\MAT{P}_{\MAT{B}}^{-1} \MAT{A}) \leq 2 \frac{\rho_{\\max}}{\gamma_{\max}} K_a$. This bound is independent  $h$ and of how the triangulation $\T_h$ intersects the boundary $\Gamma$ if we use a preconditioner $\MAT{B}_0$ with the uniform spectral equivalence $\MAT{B}_0 \sim \MAT{A}_0$ property, cf. discussion in Remark~\ref{RemMG}.

\section{Numerical experiments} \label{sectNumExp}

We choose the unit ball $\Omega:=B_1(x_0)=\{x\in\R^3:~\|x-x_0\|_2=1\}$ around midpoint $x_0\in\R^3$ and the outer domain $\Omega^*:=[-1.5,1.5]^3 \supset \Omega$. For $x\in\R^3$ we define $\hat x:= x-x_0$. If not stated differently, we use $x_0=(0.001, 0.002, 0,003)^T$ in the remainder to avoid symmetry effects. 
For the function $u:\Omega^*\to\R$, $u(x):= (3\hat x_1^2 \hat x_2 - \hat x_2^3) \exp(1-\|\hat x\|_2^2)$, the right-hand side $f(x)=u(x)(-4\|\hat x\|_2^2 + 18)$ and boundary data $g=u$ are chosen such that $u$ is a solution of \eqref{Poisson} on $\Omega$. All numerical experiments have been performed with the DROPS package \cite{DROPS}.

For the numerical discretization, the outer domain $\Omega^*$ is partitioned into $4\times 4\times 4$ cubes, where each cube is further subdivided into 6 tetrahedra, forming an initial tetrahedral triangulation $\T_0^\ast$ of $\Omega^*$. Applying an adaptive refinement algorithm, where all tetrahedra $T\in\T_0^\ast$ with $\mathrm{meas}_3(T\cap\Omega)>0$ are marked for regular refinement, we obtain the refined grid $\T_1$. Repeating this refinement process yields the grids $\T_\ell$ with refinement levels $\ell=2,\ldots,6$ and corresponding grid sizes $h_\ell=2^{-\ell}\cdot\frac{3}{4}$. 

We use linear finite elements ($k=1$) and construct finite element spaces $V_{h_\ell}$ on the respective grids $\T_\ell$, $\ell=0,1,\ldots,6$. Table~\ref{tab:grid} reports the numbers $N_0=\dim V_{h_\ell}^0$ (the number of grid points inside the fictitious domain $\Omega_h$) and $N_1=\dim V_{h_\ell}^\Gamma$ (the number of grid points on $\partial \Omega_h$) for different grid levels. We observe that $N_0$ and $N_1$ grow with the expected  factors of approximately $8$ and $4$, respectively.

\begin{table}[ht!]
\begin{minipage}{0.4\textwidth}\centering
\begin{tabular}{crr}
\toprule
$\ell$ & $N_0$ & $N_1$ \\
\midrule
 0 & 7 &	44 \\
 1 & 81 &	140 \\
 2 & 619 &	500 \\
 3 & 5,070 &	1,844 \\
 4 & 40,642 &	7,102 \\
 5 & 325,444 &	27,714 \\
 6 & 2,602,948 &	109,510 \\
\bottomrule
\end{tabular}
\caption{Dimensions $N_0, N_1$ for different refinement levels $\ell$.}
\label{tab:grid}
\end{minipage} \hfill
\begin{minipage}{0.55\textwidth}\centering
\begin{tabular}{cllll}
\toprule
$\ell$ & $\|u-u_h\|_0$ & order & $\|u-u_h\|_1$ & order \\
\midrule
 0 & 2.19E-01&&1.26E+00&     \\
 1 & 5.95E-02&1.88&6.17E-01&1.04\\
 2 & 1.43E-02&2.05&3.12E-01&0.98 \\
 3 & 3.40E-03&2.08&1.56E-01&1.00 \\
 4 & 8.15E-04&2.06&7.81E-02&1.00 \\
 5 & 1.98E-04&2.04&3.91E-02&1.00 \\
 6 & 4.89E-05&2.02&1.96E-02&1.00 \\
\bottomrule
\end{tabular}
\caption{Discretization errors w.r.t. $L^2$ and $H^1$ norm for different refinement levels $\ell$.}
\label{tab:conv}
\end{minipage}
\end{table}

Choosing the Nitsche parameter $\gamma=10$ and ghost penalty parameter $\beta=0.1$, we obtain numerical solutions $u_{h_\ell}\in V_{h_\ell}$ of the discrete problem \eqref{discrete}, with  discretization errors w.r.t. the $L^2$ and $H^1$ norm as  in Table~\ref{tab:conv}. We clearly observe optimal convergence rates in the $L^2$  and in the $H^1$ norm. 

In the following, we apply a preconditioned conjugate gradient (PCG) method  to the linear system \eqref{linsystem} and examine different choices of preconditioners $\MAT{P}\in\R^{N \times N}$. Starting with $
\bu^0=0$, the PCG iteration is stopped when the preconditioned residual is reduced by a factor $\mathrm{tol}=10^{-6}$, i.e. 
\begin{align*}
  \|\MAT{P}^{-1}(\MAT{A}\bu^k - \textbf{b})\|_2 &\leq \mathrm{tol}\, \|\MAT{P}^{-1}(\MAT{A}\bu^0 - \textbf{b})\|_2,
\end{align*}
with $\|\cdot\|_2$ the Euclidean norm.
Let $\MAT{A}=\MAT{D}+\MAT{L}+\MAT{L}^T$ be the splitting of $\MAT{A}$ into the diagonal part $\MAT{D}$ and the strict lower and upper parts $\MAT{L}, \MAT{L}^T$, respectively. Instead of the Jacobi method, which is considered in the theoretical analysis in section~\ref{sectPrecond}, we use the symmetric Gauss-Seidel preconditioner $\MAT{P}_\text{SGS}=(\MAT{D}+\MAT{L})\MAT{D}^{-1}(\MAT{D}+\MAT{L}^T)$, because this method typically is more efficient than the Jacobi method. For the matrix blocks $\MAT{A}_1 \in\R^{N_1\times N_1}$,  the corresponding symmetric Gauss-Seidel preconditioner is denoted by $\MAT{B}_{\text{SGS},1}$.

We present results for the following preconditioners:
\begin{itemize}
 \item the symmetric Gauss-Seidel preconditioner $\MAT{P}_\text{SGS}$,
 \item the block Jacobi preconditioners 
\begin{align} 
 \MAT{P}_\MAT{A}=\begin{pmatrix}
         \MAT{A}_0 & \mathbf{0} \\ \mathbf{0} & \MAT{A}_1 
        \end{pmatrix}, ~
 \MAT{P}_\MAT{D}=\begin{pmatrix}
         \MAT{A}_0 & \mathbf{0} \\ \mathbf{0} & \MAT{B}_{\text{SGS},1} 
        \end{pmatrix}, ~
 \MAT{P}_\MAT{B}=\begin{pmatrix}
         \MAT{B}_0 & \mathbf{0} \\ \mathbf{0} & \MAT{B}_{\text{SGS},1} 
        \end{pmatrix},
\end{align}
where $\MAT{B}_0$ denotes one iteration of an algebraic multigrid solver (HYPRE BoomerAMG \cite{HYPRE}).
\end{itemize}
The condition numbers $\kappa_2(\MAT{A})=\|\MAT{A}\|_2\|\MAT{A}^{-1}\|_2$ and PCG iteration numbers for different refinement levels $\ell$ are reported in Table~\ref{tab:pc-gridref}.
\begin{table}[ht!]  \centering
\begin{tabular}{clllll}
\toprule
$\ell$ & $\kappa_2(\MAT{A})$ &  \multicolumn{4}{c}{PCG iterations}\\
&& $\MAT{P}_\text{SGS}$ & $\MAT{P}_\MAT{A}$ & $\MAT{P}_\MAT{D}$ & $\MAT{P}_\MAT{B}$ \\
\midrule
 0 & 1.41E+02 &   8 &  9  & 11 & 11  \\
 1 & 1.03E+02 &   9 & 12  & 12 & 12\\
 2 & 1.58E+02 &  13 & 11  & 14 & 14\\
 3 & 2.97E+02 &  20 & 13  & 16 & 16\\
 4 & 7.74E+02 &  34 & 13  & 17 & 17\\
 5 & 3.11E+03 &  56 & 13  & 14 & 15\\
 6 & 1.26E+04 &  107 & 16 & 22 & 23 \\
\bottomrule
\end{tabular}
\caption{Condition numbers and PCG iteration numbers for different preconditioners and varying grid refinement levels $\ell$.}
\label{tab:pc-gridref}
\end{table}

For finer grid levels $\ell\geq 4$ the condition number $\kappa_2(\MAT{A})$ behaves like $\sim h^{-2}$ similar to  stiffness matrices for standard Poisson discretizations. 
For the symmetric Gauss-Seidel preconditioner $\MAT{P}_\text{SGS}$, on the finer grid levels the iteration numbers grow approximately like $h^{-1}$.
For the block preconditioners $\MAT{P}_\MAT{A}, \MAT{P}_\MAT{D}, \MAT{P}_\MAT{B}$, we observe almost constant iteration numbers for increasing level $\ell$. For all three block preconditioners the number of iterations roughly doubles when going from the coarsest level $\ell=0$ to the finest one $\ell=6$. The third preconditioner, $\MAT{P}_\MAT{B}$, is the only one with computational costs $\mathcal{O}(N)$, $N:=N_0+N_1$, with a constant independent of $\ell$.  Note that on level $\ell=5$ there is, compared to level $\ell=4$, a reduction in the number of iterations for the preconditioners  $\MAT{P}_\MAT{D}, \MAT{P}_\MAT{B}$. This might be due to a fortuitous geometric cut of $\Gamma$ with the triangulation that leads to a relatively very good performance of the symmetric Gauss-Seidel preconditioner $\MAT{B}_{\text{SGS},1}$ for the $\MAT{A}_1$ block.

We now fix the grid refinement level $\ell=3$ and vary the midpoint $x_0= (\delta, 2\delta, 3\delta)$ of the ball $\Omega$ with $\delta\in[0,0.5]$, leading to different relative positions of $\Gamma$ within the background mesh $\T_3$. The condition numbers and PCG iteration numbers for different choices of $\delta$ are reported in Table~\ref{tab:pc-relpos}.
\begin{table}[ht!]\centering
\begin{tabular}{lllllll}
\toprule
$\delta$ & $\kappa_2(\MAT{A})$ &  \multicolumn{4}{c}{PCG iterations}\\
&& $\MAT{P}_\text{SGS}$ & $\MAT{P}_\MAT{A}$ & $\MAT{P}_\MAT{D}$ & $\MAT{P}_\MAT{B}$ \\
\midrule
 0    & 2.50E+02 &  20 & 12  & 15 & 15  \\
 0.01 & 1.93E+02 &  20 & 12  & 14 & 14\\
 0.02 & 5.57E+02 &  20 & 12  & 19 & 19\\
 0.03 & 5.38E+02 &  20 & 12  & 16 & 16\\
 0.04 & 5.61E+02 &  20 & 12  & 17 & 17\\
 0.05 & 5.84E+02 &  20 & 12  & 20 & 20\\
\bottomrule
\end{tabular}
\caption{Condition numbers and PCG iteration numbers for different preconditioners and varying midpoint $x_0= (\delta, 2\delta, 3\delta)$.}
\label{tab:pc-relpos}
\end{table}
We observe that for varying $\delta$,  due to the ghost penalty stabilization, the condition number $\kappa_2(\bA)$ has the same order of magnitude. The PCG iteration numbers for the preconditioners $\MAT{P}_\text{SGS}$, $\MAT{P}_\MAT{A}$ are constant for varying $\delta$, whereas for the preconditioners $\MAT{P}_\MAT{D}$, $\MAT{P}_\MAT{B}$ the iteration numbers are  identical and change only slightly for varying $\delta$.

Finally we note that for our analysis to be applicable it is essential that we consider the fictitious domain Nitsche method \emph{with stabilization}, i.e., $\beta >0$ in \eqref{defah}. We performed numerical experiments with $\beta =0$. The results (not shown here) revealed not only that for $\beta=0$ the condition numbers $\kappa_2(\MAT{A})$ can be extremely large (due to ``bad cuts''), but also that the condition numbers  $\kappa_2(\MAT{P}^{-1}\MAT{A})$ with $\MAT{P} \in \{ \MAT{P}_\MAT{A}, \MAT{P}_\MAT{D},\MAT{P}_\MAT{B}\}$, show a very irregular behavior, where for certain cases these condition numbers become very large. From this we conclude that for good performance of the block Jacobi  preconditioners presented in this paper the use of a (ghost penalty) stabilization in the discretization method is essential.

\bibliographystyle{siam}
\bibliography{literatur}


\end{document}